\documentclass[11pt]{article}
\usepackage{color,latexsym,amsthm,amsmath,amssymb,path,enumitem,url}

\newcommand{\ignore}[1]{}

\newtheorem{theorem}{Theorem}[section]
\newtheorem{lemma}[theorem]{Lemma}
\newtheorem{corollary}[theorem]{Corollary}

\newcommand{\Proof}[1]
        {
        \noindent
        \emph{Proof #1.}~
        }
\newsavebox{\smallProofsym}                     
\savebox{\smallProofsym}                        %
        {
        \begin{picture}(7,7)                    %
        \put(0,0){\framebox(6,6){}}             %
        \put(0,2){\framebox(4,4){}}             %
        \end{picture}                           %
        }                                       %
\newcommand{\smalleop}[1]
        {
        \mbox{} \hfill #1~~\usebox{\smallProofsym}\!\!\!\!\!\!\
        }

\newcommand{\parag}[1]{\vspace{2mm}

\noindent{\bf #1} }

\usepackage{graphicx,psfrag}
\usepackage[rflt]{floatflt}

\newcommand{\placefig}[2]
        {\includegraphics[width=#2]{#1.eps}}
\usepackage{dsfont}

\newcommand{\ZZ}{\ensuremath{\mathbb Z}}
\newcommand{\RR}{\ensuremath{\mathbb R}}

\newcommand{\pts}{\mathcal P}

\DeclareMathOperator*{\EE}{\mathbb{E}}

\def\eps{{\varepsilon}}

\begin{document}
\pagenumbering{arabic}

\title{Local Properties in Colored Graphs, Distinct Distances, and Difference Sets}

\author{
Cosmin Pohoata\thanks{California Institute of Technology, Pasadena, CA, USA
{\sl apohoata@caltech.edu}.}
\and
Adam Sheffer\thanks{Department of Mathematics, Baruch College, City University of New York, NY, USA.
{\sl adamsh@gmail.com}. Supported by NSF grant DMS-1710305}}

\maketitle

\begin{abstract}
We study Extremal Combinatorics problems where local properties are used to derive global properties.
That is, we consider a given configuration where every small piece of the configuration satisfies some restriction, and use this local property to derive global properties of the entire configuration.
We study one such Ramsey problem of Erd\H os and Shelah, where the configurations are complete graphs with colored edges and every small induced subgraph contains many distinct colors.
Our bounds for this Ramsey problem show that the known probabilistic construction is tight in various cases.
We study one Discrete Geometry variant, also by Erd\H os, where we have a set of points in the plane such that every small subset spans many distinct distances.
Finally, we consider an Additive Combinatorics problem, where we are given sets of real numbers such that every small subset has a large difference set.

We derive new bounds for all of the above problems.
Our proof technique is based on introducing a variant of additive energy, which is based on edge colors in graphs.
\end{abstract}

\section{Introduction}

Erd\H os and Shelah \cite[Section V]{Erdos74} initiated the study of the following problem.
For positive integers $n,k,\ell$, we consider edge colorings of the complete graph $K_n$ such that every induced subgraph over $k$ vertices contains at least $\ell$ colors.
Let $f(n,k,\ell)$ be the minimum possible number of colors in a coloring of $K_n$ with this restriction.
As a first example, note that in $f(n,3,3)$ we consider edge colorings of $K_n$ where every triangle contains three distinct colors.
In this case no vertex can be adjacent to two edges with the same color, which immediately implies $f(n,3,3)\ge n-1$.

As another example, for any $k\ge 4$ we have
\begin{equation} \label{eq:TrivialColor}
f\left(n,k,\binom{k}{2}-\lfloor k/2 \rfloor +2 \right) = \Theta\left(n^2\right).
\end{equation}
Indeed, in this case every color can occur at most $\lfloor k/2 \rfloor -1$ times, since otherwise we will have an induced subgraph with $k$ vertices and at most $\binom{k}{2}-\lfloor k/2 \rfloor +1$ colors.
Since each color repeats a constant number of times, there are $\Theta\left(n^2\right)$ distinct colors.

One motivation for studying the above problem is that it can be seen as a variant of Ramsey's theorem (for example, see \cite[Chapter 27]{Jukna11}).
For fixed positive integers $c,k,$ and $\ell$, Ramsey's theorem studies the maximum $n$ satisfying that there exists an edge coloring of $K_n$ with $c$ colors where every $k$ of the vertices span at least $\ell=2$ distinct colors.
In the current problem, we allow $\ell$ to be larger than two, fix the value of $n$, and look for the minimum $c$ satisfying the above.

A few first results for the problem were obtained by Erd\H os, Elekes, and F\"uredi \cite[Section 9]{Erdos81}.
Erd\H os and Gy\'arf\'as \cite{EG97} started studying the problem more systematically.
Considering a restriction slightly weaker than in \eqref{eq:TrivialColor}, they derived the bound
\begin{equation}\label{eq:OneFromTriv}
 f\left(n,k,\binom{k}{2}-\lfloor k/2 \rfloor +1\right) = \Omega\left(n^{4/3}\right). 
\end{equation}
That is, the boundary between a trivial problem and a long-standing open problem passes between $\ell \ge \binom{k}{2}-\lfloor k/2 \rfloor +2$ and $\ell \le \binom{k}{2}-\lfloor k/2 \rfloor +1$.
Using a probabilistic construction, Erd\H os and Gy\'arf\'as \cite{EG97} derived the bound
\begin{equation} \label{eq:ProbLower}
f\left(n,k,\ell\right) = O\left(n^{\frac{k-2}{\binom{k}{2}-\ell+1}}\right).
\end{equation}

S\'ark\"ozy and Selkow \cite{SS01} proved that for every $k\ge 6$ there exists $\eps>0$ such that
\[ f\left(n,k,\binom{k}{2}-k +\lceil \log k\rceil +1\right) = \Omega\left(n^{1+\eps}\right). \]

Erd\H os and Gy\'arf\'as also proved that $f\left(n,k,k\right)$ is always polynomial in $n$.
Conlon, Fox, Lee, and Sudakov \cite{CFLS14} showed that this is tight by proving that $f\left(n,k,k-1\right)$ is subpolynomial in $n$ for every $k\ge 4$.
Axenovich, F\"uredi, and Mubayi \cite{AFM00} studied a bipartite variant of the problem.
These are just some of main results in the study of the problem, and far from being an exhaustive list.

The following theorem is the main result of the current work.
\begin{theorem} \label{th:LocalProp}
For any integers $k > m \ge 2$,
\[ f\left(n,k,\binom{k}{2} -  m\cdot \left\lfloor\frac{k}{m+1}\right\rfloor  + m+1 \right) = \Omega\left(n^{1+\frac{1}{m}}\right). \]
\end{theorem}

For example, by applying Theorem \ref{th:LocalProp} with $m=2$ we get
\[ f\left(n,k,\binom{k}{2} - 2\cdot \left\lfloor\frac{k}{3}\right\rfloor + 3 \right) = \Omega\left(n^{3/2}\right). \]
In \eqref{eq:OneFromTriv}, Erd\H os and Gy\'arf\'as derived a bound of $\Omega(n^{4/3})$ colors when $\ell$ is one unit away from the trivial case.
Theorem \ref{th:LocalProp} implies this bound already when $\ell \ge \binom{k}{2} - 3\cdot \left\lfloor k/4\right\rfloor + 4$.

The bound of Theorem \ref{th:LocalProp} is asymptotically tight up to sub-polynomial factors for every $m\ge 2$, except possibly for small values of $k$.
Indeed, for every $\eps>0$ the bound of \eqref{eq:ProbLower} implies
\[ f\left(n,k,\binom{k}{2} -  m\cdot \left\lfloor\frac{k}{m+1}\right\rfloor  + m+1 \right) = O\left(n^{1+\frac{1}{m}+\eps}\right), \]
for every sufficiently large $k$.

\parag{Distinct distances with local properties. }
The \emph{Erd\H os distinct distances problem} is a main problem in Discrete Geometry.
This problem asks for the minimum number of distinct distances spanned by a set of $n$ points in $\RR^2$.
That is, denoting the distance between two points $p, q \in \RR^2$ as $|pq|$, the problem asks for $\min_{|\pts|=n} |\{|pq| :\  p, q \in \pts \}|$.
Note that $n$ equally spaced points on a line span $n-1$ distinct distances.
Erd\H os \cite{erd46} observed that a $\sqrt{n}\times \sqrt{n}$ section of the integer lattice $\ZZ^2$ spans $\Theta(n/\sqrt{\log n})$ distinct distances.
Proving that every point set determines at least some number of distinct distances turned out to be a deep and challenging problem.

The above problem is just one out of a large family of distinct distances problems, including higher-dimensional variants, structural problems, and many other types of problems (for example, see \cite{Sheffer14}).
The main problems in this family were proposed by Erd\H os and have been studied for decades.
After over 60 years and many works on distinct distances problems, Guth and Katz \cite{GK15} almost settled the original question by proving that every set of $n$ points in $\RR^2$ spans $\Omega(n/ \log n)$ distinct distances.
Surprisingly, so far this major discovery was not followed by significant progress in the other main distinct distances problems.

Let $\phi(n,k,l)$ denote the minimum number of distinct distances that are determined by a planar $n$ point set $\pts$ with the property that any $k$ points of $\pts$ determine at least $l$ distinct distances. That is, by having a local property of every small subset of points, we wish to obtain a global property of the entire point set.

For example, the value of $\phi(n,3,3)$ is the minimum number of distinct distances that are determined by a set of $n$ points that do not span any isosceles triangles (including degenerate triangles with three collinear vertices).
Since no isosceles triangles are allowed, every point determines $n-1$ distinct distances with the other points of the set, and we thus have $\phi(n,3,3) = \Omega(n)$.
Erd\H os \cite{Erdos86} observed the following upper bound for $\phi(n,3,3)$.
Behrend \cite{Behrend46} proved that there exists a set $A$ of positive integers $a_1< a_2 < \cdots < a_n$ such that no three elements of $A$ determine an arithmetic progression and $a_n < n2^{O(\sqrt{\log n})}$.
Therefore, the point set $\pts_1 = \{(a_1,0), (a_2,0),\ldots, (a_n,0)\}$ does not span any isosceles triangles.
Since $\pts_1 \subset \pts_2 = \{(1,0),(2,0),\ldots,(a_n,0) \}$ and $D(\pts_2)< n2^{O(\sqrt{\log n})}$, we have $\phi(n,3,3) < n2^{O(\sqrt{\log n})}$.

Recently, Fox, Pach, and Suk \cite{FPS17} showed that for any $\eps>0$
\[ \phi\left(n,k,\binom{k}{2}-k+6\right) = \Omega\left(n^{8/7-\eps}\right).\]
The proof of this result is based on the geometry of the problem, and thus does not extend to the more general variant of $f(n,k,\ell)$.

Upper bounds for $f(n,k,\ell)$ do not immediately apply to $\phi(n,k,\ell)$.
For example, the bounds of Conlon, Fox, Lee, and Sudakov \cite{CFLS14} are false for the distances case (these bounds are subpolynomial in $n$ while every $n$ points in $\RR^2$ span $\Omega(n/\log n)$ distinct distances).
On the other hand, any lower bound for $f(n,k,\ell)$ remains valid for $\phi(n,k,\ell)$.
In particular, the following is an immediate corollary of Theorem \ref{th:LocalProp}.
\begin{corollary} \label{co:LocalDist}
For any integers $k > m \ge 2$,
\[ \phi\left(n,k,\binom{k}{2} -  m\cdot \left\lfloor\frac{k}{m+1}\right\rfloor  + m+1 \right) = \Omega\left(n^{1+\frac{1}{m}}\right). \]
\end{corollary}
To prove this corollary one can build a complete graph where every point is a vertex, and every distance corresponds to a distinct edge color.

Corollary \ref{co:LocalDist} leads to a better bound than the one of Fox, Pach, and Suk \cite{FPS17} when $\ell \ge \binom{k}{2} - 7\cdot \lfloor k/8\rfloor + 8$.

While there are many problems in which the conjectured number of distinct distances is $\Omega(n^{2-\eps})$, we are very far from deriving this bound for any of those.
For example, see \cite{Sheffer14}.
As far as we know, the above is the first case where a bound stronger than $\Omega(n^{4/3})$ is obtained for a non-trivial distinct distances problem.

\parag{Difference sets with local properties.}
Zeev Dvir recently suggested studying the following Additive Combinatorics variant of the local properties problem.
Given a finite $A\subset \RR$, the \emph{difference set} of $A$ is
\[ A-A = \{a-a'\ :\ a,a'\in A \}. \]

For positive integers $n,k,\ell$, we consider sets $A\subset \RR$ of size $n$ such that every subset $A'\subset A$ of size $k$ satisfies $|A'-A'|\ge \ell$.
Let $g(n,k,\ell)$ be the minimum size of $A-A$ among all sets $A$ that satisfy the above restriction.
For simplicity we will ignore non-positive differences.
For example, when considering sets with no 3-term arithmetic progression we will write $g(n,3,3)$ instead of $g(n,3,7)$ (we ignore 0 and three negative differences).
This notation does not change the problem, and is somewhat more intuitive.

While this seems like an interesting natural Additive Combinatorics problem, we only managed to find one minor and brief mention of it.
It is stated in \cite[Section 9]{EG97} that Erd\H os and S\'os proved $g(n,4,5) \ge \binom{n}{2}-n+2$, although it seems that this was never published.
The following is a corollary of Theorem \ref{th:LocalProp}.
\begin{corollary}
For any integers $k > m \ge 2$,
\[ g\left(n,k,\binom{k}{2} - m\cdot \left\lfloor\frac{k}{m+1}\right\rfloor  + m+1 \right) = \Omega\left(n^{1+\frac{1}{m}}\right). \]
\end{corollary}
To prove this corollary one can build a complete graph where every element of $A$ is a vertex, and every difference corresponds to a distinct edge color.

We now present an example illustrating that $g(n,k,\ell)$ and $\phi(n,k,\ell)$ may be very different problems.
Currently, nothing is known about $\phi(n,4,5)$ beyond the trivial bounds $\phi(n,4,5)=\Omega(n)$ and $\phi(n,4,5)=O(n^2)$ --- this is considered to be a difficult open problem.
On the other hand, it can be easily shown that $g(n,4,5) = \Theta(n^2)$.
Indeed, consider a set $A$ of $n$ real numbers and with every $A' \subset A$ of size four satisfying $|A'-A'|\ge 5$.
Assume that there exist four distinct reals $a_1,a_2,a_3,a_4\in A$ such that $a_1-a_2=a_3-a_4$.
This implies that $a_1-a_3 = a_2-a_4$, and thus these four points span at most four differences.
The above contradiction implies that no difference repeats more than twice (it is still possible that $a_1-a_2 = a_3-a_1$).
We conclude that $|A-A|=\Theta(n^2)$.

It would be interesting to further study this problem of difference sets with local properties.
For example, what are non-trivial upper bounds for $g(n,k,\ell)$? What happens when $\ell$ is much smaller than $\binom{k}{2}$?

\parag{Our proof technique.}
To prove Theorem \ref{th:LocalProp}, we define a new abstract variant of the concept of additive energy, which is a main tool in Additive Combinatorics.
Given a finite $A\subset \RR$, the \emph{sum set} of $A$ is
\[ A+A = \{a+a'\ :\ a,a'\in A \}. \]

A uniformly chosen set $A$ of a fixed finite size is expected to satisfy $|A+A|=\Theta(|A|^2)$.
On the other hand, there are sets that satisfy $|A+A|=\Theta(|A|)$, such as arithmetic progressions.
We say that such sets have an ``additive structure'' that leads to a small sum set.
The \emph{polynomial Freiman--Ruzsa }conjecture is a main open problem in Additive Combinatorics, asking to characterize the sets that have a small sum set.

One main tool for studying the additive structure of a finite $A\subset \RR$ is the \emph{additive energy} of $A$:
\[ E(A) = \left|\left\{(a,b,c,d)\in A^4 :\ a+b=c+d\right\}\right|. \]
While the size of the sum set of $A$ is at least linear in $|A|$ and at most quadratic in $|A|$, the additive energy of $A$ is at least quadratic in $|A|$ and at most cubic in $|A|$.

A small sum set implies a large additive energy.
In particular, a simple use of the Cauchy-Schwarz inequality implies $E(A)\ge \frac{|A|^4}{|A+A|}$.
In the other direction, the \emph{Balog--Szemer\'edi--Gowers theorem} implies that if $E(A)$ is large then there exists a large subset $A'\subset A$ such that $A'+A'$ is small.
For more information, see for example the book ``Additive Combinatorics'' by Tao and Vu \cite{TV06}.

To prove Theorem \ref{th:LocalProp}, we define the following abstract graph variant of additive energy.
Given a graph $G=(V,E)$ with colored edges, we denote by $c(u,v)$ the color of the edge $(u,v)\in E$.
We define the \emph{color energy} of $G$ as\footnote{We use $\EE(G)$ rather than $E(G)$ since the latter is a standard notation for the set of edges of $G$.}
\begin{equation*}
\EE(G) = \left|\left\{(v_1,u_1,v_2,u_2)\in V^{4} :\ c(v_1,u_1) = c(v_2,u_2)\right\}\right|.
\end{equation*}
That is, instead of the number of quadruples that satisfy an additive relation, we ask for the number of quadruples that satisfy a color relation.

There exist energy variants for Cayley graphs that are based on the corresponding group action (for example, see Gowers \cite{Gowers08}).
However, as far as we know this is the first use of such a non-algebraic energy variant.

\parag{Acknowledgements.}
We are indebted to the anonymous referees, who made many helpful suggestions for improving a previous draft of this work.
We would like to thank Zeev Dvir for suggesting the Additive Combinatorics variant of the problem, and to Robert Krueger for some helpful discussions.

\section{Proof of Theorem \ref{th:LocalProp}}\label{sec:DDlocal}

In this section we prove Theorem \ref{th:LocalProp}.
We begin by recalling the statement of this theorem.
\vspace{1mm}

\noindent {\bf Theorem \ref{th:LocalProp}.}
\emph{For any integers $k > m \ge 2$,}
\[ f\left(n,k,\binom{k}{2} -  m\cdot \left\lfloor\frac{k}{m+1}\right\rfloor  + m+1 \right) = \Omega\left(n^{1+\frac{1}{m}}\right). \]

To prove the theorem, we will rely on the following simple counting lemma (see \cite{Erdos64} and \cite[Lemma 2.3]{Jukna11}).

\begin{lemma} \label{le:GeneralSetIntersection}
Let $A$ be a set of $n$ elements and let $d\ge 2$ be an integer.
Let $A_1,\ldots,A_k$ be subsets of $A$, each of size at least $m$.
If $k \ge 2d n^d/m^d$ then there exist $1\le j_1 < \ldots < j_d \le k$ such that $|A_{j_1}\cap \ldots \cap A_{j_d}| \ge \frac{m^d}{2n^{d-1}}$.
\end{lemma}

\begin{proof}[Proof of Theorem \ref{th:LocalProp}.]
To prove the theorem we will prove that for any integers $a,b \ge 2$
\begin{equation} \label{eq:abFormulation}
f\left(n,a(b+1),\binom{a(b+1)}{2} - ba + b+1 \right) = \Omega\left(n^{1+\frac{1}{b}}\right).
\end{equation}
When $k$ is divisible by $m+1$, we obtain the statement of the theorem by setting $b=m$ and $a= k/(m+1)$.
When $k$ is not divisible by $m+1$, we write $a= \lfloor k/(m+1) \rfloor$, and rewrite \eqref{eq:abFormulation} as $f\left(n,k,\binom{k}{2} - ba + b+1 \right)= \Omega\left(n^{1+\frac{1}{b}}\right)$.
The rest of the proof is the same in both cases.
However, when reading the proof for the first time we recommend assuming that $k$ is divisible by $m+1$.

Let $G=(V,E)$ be a copy of $K_n$ with colored edges, such that every induced $K_{a(b+1)}$ contains at least  $\binom{a(b+1)}{2} - ba + b+1$ distinct colors.
We denote the set of colors as $C = \{c_1,c_2,\ldots\}$, and the color of an edge $(v,u)\in E$ as $c(v,u)$.
Our goal is to prove that $|C|=\Omega\left(n^{1+1/b}\right)$, and we begin by studying some configurations that cannot occur in $G$.

\parag{Forbidden configurations.}
We first show that no vertex can be adjacent to many edges of the same color.
Assume for contradiction that there exists a vertex $v \in V$ and color $c\in C$, such that at least $ba-b+1$ vertices $u\in V$ satisfy $c(v,u)=c$.
Let $V' \subset V$ consist of $v$, of $ba-b+1$ vertices satisfying $c(v,u)=c$, and of $b+a-2$ additional vertices.
Then $V'$ is a set of $a(b+1)$ vertices, and the induced subgraph on $V'$  contains at most $\binom{a(b+1)}{2} - ba + b$ distinct colors, contradicting the assumption on $G$.
This contradiction implies that for every $v\in V$ and $c\in C$, at most $ba-b$ of the edges incident to $v$ have color $c$.
This in turn implies that every color appears at most $b(a-1)n/2$ times.

We next show that there cannot be $a$ vertices that are adjacent to the same $b$ ``popular'' colors.
Let $V_c \subset V$ be the set of endpoints of edges of color $c$.
For an integer $j$, let $C_j$ be the set of colors that appear at least $2^j$ times.
For $j$ with  $2^j\ge a$, assume for contradiction that there exist $c_1, \ldots, c_b \in C_j$ that satisfy $|V_{c_1} \cap \cdots \cap V_{c_b}| \ge a$.
Let $V' = \{v_1,v_2,\ldots,v_a\}$ be a set of $a$ vertices from $V_{c_1} \cap \cdots \cap V_{c_b}$.
That is, for every vertex $v\in V'$ and every color $c \in \{c,\ldots,c_b\}$ there exists $u\in V$ satisfying $c(v,u)=c$.
An example is depicted in Figure \ref{fi:ForbiddenConf2}.

\begin{figure}[h]
\centerline{\placefig{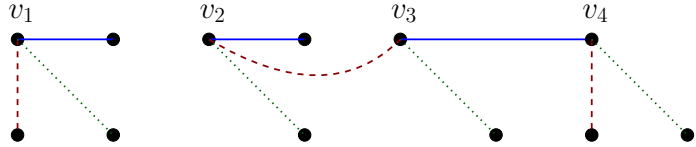}{0.6\textwidth}}
\vspace{-1mm}

\caption{\small \sf The three colors are represented as solid, dashed, and dotted edges. Every vertex of $V' = \{v_1,v_2,v_3,v_4\}$ is incident to an edge of every color. }
\label{fi:ForbiddenConf2}
\vspace{-2mm}
\end{figure}

We construct a subset $V'' \subset V$ as follows: For every $v\in V'$ and color $c \in C_j$, if there exist vertices $u \in V\setminus V'$ such that $c(v,u)=c$ then we add one such vertex to $V''$.
If a vertex $u \in V\setminus V'$ was added more than once to $V''$, we consider $V''$ as containing a single copy of $u$ (that is, $V''$ is not a multiset).
Note that if for $v\in V'$ and $c\in C_j$ there is no $u \in V\setminus V'$ satisfying $c(v,u)=c$ then there exists $u\in V'$ satisfying $c(v,u)=c$.
For $1 \le i \le b$, let $r_i$ denote the number of vertices of $V'$ that are not connected to any vertex of $V\setminus V'$ with an edge of color $c_i$.
For every $1 \le i \le b$, at least $r_i/2$ edges in the induced subgraph of $V'$ have color $c_i$.
Thus, $V^* = V' \cup V''$ is a set of at most $(b+1)a-\sum_{i=1}^b r_i$ vertices of $V$ and at least $a-r_i/2$ edges with color $c_i$.
For $1 \le i \le b$, we add $\lfloor r_i/2 \rfloor$ additional edges of color $c_i$ by adding at most $r_i$ vertices of $V$ to $V^*$.
This is always possible since $c_i$ is in $C_j$ and $2^j\ge a$.
Since the resulting set $V^*$ contradicts the assumption about $G$, no distinct $c_1, \ldots, c_b \in C_j$ satisfy $|V_{c_1} \cap \cdots \cap V_{c_b}| \ge a$.

\parag{Popular colors.}
Let $k_j = |C_j|$.
We now derive two upper bounds for $k_j$.

Fix an integer $j$ satisfying $2^j\ge b(2a^{b+1}n^{b-1})^{1/b}$ and let $c \in C_j$.
Since a single vertex is incident to at most $ba-b$ edges with color $c$, we get that $|V_c|\ge \frac{2^{j+1}}{ba-b} \ge \frac{2^j}{ba}$.
Since no $b$ sets $V_c$ have a large intersection, we use the contrapositive of Lemma \ref{le:GeneralSetIntersection} to obtain that the number of sets is not large.
By the lower bound for $2^j$, every $b$ sets $V_c$ intersect in less than $a \le \frac{2^{jb}}{2a^b b^b n^{b-1}}$ vertices.
Since the size of each such set is at least $m=\frac{2^j}{ba}$, the contrapositive of the lemma implies that the number of sets is
\begin{equation} \label{eq:RichDistUpper}
k_j < \frac{2bn^b}{m^b}= \frac{2bn^b}{(2^j/ba)^b} = \frac{2n^b b^{b+1}a^b}{2^{jb}}.
\end{equation}

When $2^j< b(2a^{b+1}n^{b-1})^{1/b}$ we will rely on the straightforward bound
\begin{equation} \label{eq:PoorDistUpper}
k_j < n^2/2^j.
\end{equation}

\parag{An energy argument.}
Let $m_c$ be the number of edges with color $c$.
In the beginning of the forbidden configurations part above, we proved that $m_c \le b(a-1)n$ for every $c\in C$.
Since every edge contributes to exactly one $m_c$, we have $\sum_{c\in C} m_c = \binom{n}{2}$.

Recall that the color energy of $G$ is defined as
\begin{equation*}
\EE(G) = \left|\left\{(v_1,u_1,v_2,u_2)\in V^{4} :\ c(v_1,u_1)=c(v_2,u_2)\right\}\right|.
\end{equation*}
In other words, $\EE(G)$ is the number of pairs of edges with the same color.
Since the graph is undirected, when computing $\EE(G)$ we consider $(v_1,u_1,v_2,u_2)$ and $(u_1,v_1,v_2,u_2)$ as the same quadruple.
On the other hand, we count $(v_1,u_1,v_2,u_2)$ and $(v_2,u_2,v_1,u_1)$ as two separate quadruples.
We can also think of $\EE(G)$ as the square of the $\ell_2$-norm of the color frequencies, since
\[ \EE(G) = \sum_{c \in C} m_c^2. \]

By the Cauchy-Schwarz inequality, we have
\[ \EE(G) = \sum_{c\in C} m_c^2 \ge \frac{\left(\sum_{c\in C} m_c\right)^2}{|C|} = \frac{\binom{n}{2}^2}{|C|} = \Omega\left( \frac{n^4}{|C|}\right). \]

Let $t = \lfloor \log b(2a^{b+1}n^{b-1})^{1/b} \rfloor$.
By dyadic pigeonholing together with \eqref{eq:RichDistUpper} and \eqref{eq:PoorDistUpper}, we obtain
\begin{align*}
\EE(G) = \sum_{c\in C} m_c^2 &= \sum_{j=0}^{\log (ban)} \sum_{c\in C \atop 2^{j} \le m_c <2^{j+1} } m_c^2 < \sum_{j=0}^{\log (ban)} k_j \left(2^{j+1}\right)^2 \\
&= \sum_{j=0}^{t} k_j 2^{2j+2} + \sum_{j=t+1}^{\log (ban)} k_j 2^{2j+2} \\[2mm]
&= \sum_{j=0}^{t} \frac{n^2}{2^j} \cdot 2^{2j+2} + \sum_{j=t+1}^{\log (ban)} \frac{2n^b b^{b+1}a^b}{2^{jb}} \cdot  2^{2j+2} \\[2mm]
&= O\left(n^{3-1/b}\right)+O\left(n^{3-2/b}\log n\right) = O\left(n^{3-1/b}\right).
\end{align*}

The $\log n$ in the last line exists only when $b=2$, but it does not affect the final bound in any case.
Combining the two above bounds on $\EE(G)$ yields $|C|=\Omega(n^{1+1/b})$, as required.
\end{proof}

\parag{Remark. } One way to improve the proof of Theorem \ref{th:LocalProp} might be to derive an upper bound on $k_j$ stronger than the straightforward bound of \eqref{eq:PoorDistUpper} when $2^j\approx n^{(b-1)/b}$.


\end{document}